\documentclass{amsart}
\usepackage{amsmath,amssymb, xypic,verbatim,amscd,color}
\usepackage{graphicx}
\usepackage{colordvi}

\usepackage{amsfonts, amsmath, amssymb, amsthm}

\newtheorem{theorem}{Theorem}[section]
\newtheorem{remark}[theorem]{ Remark}

\newtheorem{proposition}[theorem]{Proposition}

\newtheorem{lemma}[theorem]{Lemma}

\newtheorem{definition}[theorem]{Definition}

\newcommand{\V}{\Vert}

\newcommand{\NN} {\mathbb N}

\newcommand{\pa} {\partial}
\newcommand{\Cal} {\mathcal}

\newcommand{\beq} {\begin{equation}}
\newcommand{\eeq} {\end{equation}}

\begin{document}
\title[Evolution of spectra ]{ Variation of Laplace spectra of compact ``nearly'' hyperbolic surfaces}

\author{Mayukh Mukherjee}

\address{Max Planck Institute for Mathematics\\ Vivatsgasse 7\\ 53111 Bonn,
\\ Germany}

\email{mukherjee@mpim-bonn.mpg.de}
\begin{abstract} We use the time real analyticity of Ricci flow proved by Kotschwar to extend a result in ~\cite{B}, namely, we prove that the Laplace spectra of negatively curved compact surfaces having same genus $\gamma \geq 2$, same area and same curvature bounds vary in a ``controlled way'', of which we give a quantitative estimate (Theorem \ref{MT} below). We also observe how said real analyticity can lead to unexpected conclusions about spectral properties of generic metrics on a compact surface of genus $\gamma \geq 2$ (Proposition \ref{AV} below). 
\end{abstract}
\maketitle
\section{Normalized Ricci flow and evolution of spectrum}
\subsection{Introduction}
Consider two compact negatively curved surfaces $M_1$ and $M_2$ of same genus $\gamma \geq 2$ and same area $A$ such that their scalar curvatures $R_1$ and $R_2$ (respectively) satisfy $\alpha < R_i < \beta < 0$ for $i = 1, 2$. 
Now, consider the compact surfaces $N_i$ obtained by scaling the metric on $M_i$, so that $N_i$ have average scalar curvature $-2$. Let $N_i$ flow to hyperbolic surfaces $S_i$ in the limit under the normalized Ricci flow. In this note, 
our chief aim is to prove the following:
\begin{theorem}\label{MT}
Let $\delta$ be the distance between $S_1$ and $S_2$ in the Teichm\"{u}ller space $\Cal{T}_g$. Then, we have
\[
e^{-\frac{\alpha}{r} - \frac{r}{\beta} - 4\delta}\lambda_n(M_2) \leq \lambda_n(M_1) \leq e^{\frac{\alpha}{r} + \frac{r}{\beta} + 4\delta}\lambda_n(M_2), n \in \NN,\]
where $r$ denotes the average scalar curvatures of $M_i, i = 1,2$ and $\lambda_n (M)$ represent the Laplace eigenvalues of $M$.
\end{theorem}
For a definition of the distance in the Teichm\"{u}ller space alluded to above, see Definition \ref{dis} in Subsection \ref{dist}.
\subsection{Background: Ricci flow facts}
The Ricci flow program was introduced by Hamilton in ~\cite{H}; the main idea is to take an initial metric and flow it into ``nicer'' metrics, at least with respect to curvature properties, according to the equation 
\beq\label{RF}
\pa_t g = - 2 \text{Ric} g,\eeq
where $g$ denotes the (time-dependent) metric on $M$, and $\text{Ric}$ the Ricci curvature tensor associated to said metric. Hamilton and later DeTurck (~\cite{D}) provided proofs of short time existence of (\ref{RF}). It has also been proved that the Ricci flow (henceforth abbreviated RF) on a compact manifold can continue as long as the Riemannian curvature tensor does not explode. To balance for this blow-up phenomenon, one sometimes uses the so-called ``normalized Ricci flow'' (henceforth abbreviated NRF), which is defined by
\beq\label{NRF}
\pa_t g = -2\text{Ric} g + \frac{2}{n}rg,
\eeq
where $R$ is the scalar curvature and $r = \frac{1}{\text{Vol }M}\int_M R dV$ is the average scalar curvature of the manifold $M$ of dimension $n$. One main difference between (\ref{RF}) and (\ref{NRF}) is that (\ref{NRF}) rescales the volume of the manifold at every step, so that the volume remains constant throughout the flow. Using the NRF on surfaces, Hamilton proved the following
\begin{theorem}\label{Hamil}
If $(M, g_0)$ is a closed Riemannian surface, then there exists a unique solution $g_t$ of the NRF
\beq
\pa_t g = (r - R)g, \text{   } g(0) = g_0.
\eeq
The solution exists for all time. As $t \to \infty$, the metrics $g_t$ converge uniformly in any $C^k$-norm to a smooth metric $g_\infty$ of constant curvature.
\end{theorem}
If $T$ is the maximal time for existence of (\ref{RF}), by a well-known result of Bando (see ~\cite{Ba}), a solution $(M, g(t))$, $t \in (0, T]$ is real analytic in space when $M$ is compact. This was improved upon by Kotschwar (~\cite{K}), who provided sufficient conditions for a solution $(M, g(t))$ to be real analytic in both space and time when $(M, g_0)$ is complete. He proved
\begin{theorem}\label{Kot}
Suppose $(M, g_0)$ is complete and $g(t)$ is a smooth solution to (\ref{RF}) satisfying 
\beq\label{Kc}
\sup_{M \times [0, \Omega]}|\text{Rm}(x, t)| \leq C.
\eeq
Then the map $g : (0, \Omega) \to X$ is real-analytic where $X$ denotes the Banach space $BC(T_2(M))$ equipped with the supremum norm $\V .\V_{g(0)}$ relative to $g(0)$.
\end{theorem}
It is a natural question to ask what happens to the spectrum of the Laplacian under the Ricci flow.  Theorem \ref{Kot}, in conjunction with Kato's analytic perturbation theory, tells us that the eigenvalues and eigenfunctions of the Laplacian vary real analytically in time as long as (\ref{Kc}) is satisfied. Let us also remark here that without Kotschwar's result, we are assured of twice differentiability of the eigenvalues from very general perturbative arguments. For more details on the variation of eigenvalues and eigenvectors of a one-parameter family of unbounded self-adjoint operators with common domain of definition and compact resolvent on a variety of regularity scales, see ~\cite{KMR}.\newline 
We note that Kotschwar's result is valid apriori for Ricci flows. But using the correspondence between RF and NRF, we can establish that the NRF is real analytic in time if the RF is. If the NRF is written as 
\beq\label{NRF1}
\pa_\tau \tilde{g}  = - 2\text{Ric} \tilde{g} + \tilde{r}\tilde{g},\eeq
then the relation between (\ref{NRF1}) and (\ref{RF}) is given by 
\beq\label{tau}
\tilde{g}(\tau) = \frac{g(t(\tau))}{\text{Area}_{t(\tau)}M},\eeq
where 
\beq\label{tau1}
t(\tau) = \frac{\text{Area}_0 M}{4\pi \chi}(1 - e^{-4 \pi \chi\tau}), \text{    Area}_tM = \text{Area}_{0}M - 4\pi\chi t,
\eeq
$\chi$ being the Euler characteristic of the surface. \newline
Now, using (\ref{tau1}), we see that the right hand side of (\ref{tau}) is real analytic whenever $g(t)$ is real analytic in $t$. Also, as far as we are on a compact surface of genus $\gamma \geq 2$, we know that (\ref{Kc}) is satisfied. Which leads us to conclude that
\begin{lemma}\label{con}
Given a closed surface $(M, g_0)$, where the genus $\gamma$ of $M$ is $\geq 2$, the NRF on $M$ exists for all time, and the solution is real-analytic in space and in the time variable $t \in (0, \infty)$.
\end{lemma}
\subsection{Avoidance of specific numbers as eigenvalues}
Let us see how Lemma \ref{con} allows the generic metric on a closed surface of genus $\geq 2$ to avoid large predetermined real numbers as Laplace eigenvalues. Let us fix the genus $\gamma$ of the compact surfaces under consideration. Let us also fix the area $A$ of the surfaces to be equal to $ -2\pi\chi$, so that the unique metric of constant curvature in each conformal class has Gaussian curvature $-1$ (that is, $R = r = -2$). Then, the space of all Riemannian metrics of area $-2\pi\chi$ with the usual Whitney $C^\infty$ topology forms a smooth trivial bundle over all hyperbolic metrics, which are the unique constant curvature metrics in each conformal class. Using the NRF, we have a real analytic path $g_t$ starting from any member $g_0$ of a given fiber of the bundle, and ending (as $t \to \infty$) in a hyperbolic metric.  Also, two distinct real analytic NRF paths will never intersect, by backwards uniqueness of Ricci flows (see ~\cite{K1}). By Kato's perturbation theory, the Laplace eigenvalues $\lambda_n(t)$ and eigenfunctions $\varphi_t$ vary real analytically along $g_t$ as well. Now, let $\Cal{M}(M)$ denote the space of all metrics on $M$ with area normalized as above, and let $\Cal{H}(M)$ denote the space of all hyperbolic metrics on $M$. Consider a positive real number $\sigma$. We have the following
\begin{proposition}\label{AV}
If $\sigma > \frac{1}{4}$, then the subset of metrics $g \in \Cal{M}(M)$ such that $\sigma \in \text{Spec }(M, g)$ has zero interior.
\end{proposition}
\begin{proof}
Suppose to the contrary. Assume the existence of $g_0 \in \Cal{M}(M)$ such that all metrics in a neighbourhood $U$ of $g_0$ have $\sigma$ as an eigenvalue. Then, picking any $g \in U$, we have a time real analytic NRF path $g_t$ through $g$ such that, say, $g = g_{t_0}$ and $\sigma \in \text{Spec }(M, g_t)$ for $t$ close to $t_0$. Since NRF preserves the conformal class, passing to the limit, our assumption produces an open set $V$ in $\Cal{H}(M)$ such that each metric in $V$ will have $\sigma$ as an eigenvalue. If we can prove that the hyperbolic metrics having $\sigma$ as an eigenvalue have zero interior in the space $\Cal{H}(M)$, then we will see that the subset of metrics having $\sigma$ as an eigenvalue has empty interior in $\Cal{M}(M)$. This we can argue when $\sigma > 1/4$. Let $h \in V \subset \Cal{H}(M)$, and consider a real analytic path $h_t$ of hyperbolic metrics through $h$ such that $h = h_s$, say, where $s > 0$, and $h_t$ has shortest closed geodesic of length $t$ (from ~\cite{W1}, we know that we can choose such a real analytic path with shrinking geodesics all the way up to $t \searrow 0$). 
Now, $\lim_{t \to 0} \sigma (t) = \sigma > \frac{1}{4}$, because $\sigma (t)$ is constant in $(0, s)$ by real analyticity. 
By Theorem 5.14 of ~\cite{W2}, a real analytic path of eigenvalues whose limit is greater than $\frac{1}{4}$ and which is associated to such a real analytic path of hyperbolic surfaces with geodesics pinched in the limit, must vary non-trivially, that is, it cannot be a constant path. This contradicts the previous conclusion that $\sigma (t)$ is constant for $t \in (0, s)$, and proves the proposition.
\end{proof}
\section{Spectral evolution formula under nrf and conclusions}
Let us now calculate a formula for variation of spectrum for a compact surface of genus $\gamma \geq 2$ under the NRF. 
For previous literature on this kind of investigation, see ~\cite{Di}, and work by Cao, for example, ~\cite{C1} and ~\cite{C2} and references therein (however, also compare Remark \ref{Sale} below). 
Here is our 
\begin{lemma}\label{2.1}
Under the NRF acting on a closed surface $(M, g_0)$ with (time-independent) area $1$, the eigenvalues evolve according to the following formula:
\beq\label{ev}
\frac{d\lambda}{dt} = \lambda\int_M R\varphi^2 dV - r\lambda,
\eeq
where $\varphi$ is a (spatially smooth temporally analytical) eigenfunction corresponding to $\lambda$ which is normalized, that is, $\int_M \varphi dV = 0$, and $\V \varphi\V_{L^2} = 1$\footnote{We use $dV$ to denote the area element.}.
\end{lemma}
\begin{proof}
We begin by noting that the evolution of the volume form (in this case area form) under the NRF is given by $\frac{d}{dt}dV = (r - R)dV$ (see Lemma 3.9, Chapter 3 of ~\cite{CK}). 
Differentiating  $\int_M \varphi dV = 0$ with respect to time, we get $\int_M \pa_t\varphi dV = \int_M \varphi RdV$. Also, $\V \varphi\V_{L^2} = 1$ implies that $2\int_M \varphi\pa_t\varphi dV = \int_M \varphi^2 RdV  - r$. \newline
Now, since $\lambda = \V \nabla \varphi\V^2$, we have,
\begin{align*}
\frac{d\lambda}{dt} & = \frac{d}{dt}\V\nabla\varphi\V^2 = \int_M (\frac{d}{dt}|\nabla\varphi|^2) dV + \int_M |\nabla \varphi|^2\frac{d}{dt}dV\\
& = \int_M (\frac{d}{dt}g^{ij})\nabla_i \varphi\nabla_j\varphi dV + 2\int_M \langle \nabla\pa_t\varphi, \nabla \varphi\rangle dV + \int_M(r - R)|\nabla \varphi|^2dV,
\end{align*}
where $\langle \nabla \pa_t \varphi, \nabla \varphi\rangle$ denotes  inner product of $\nabla\pa_t \varphi$ and $\nabla\varphi$ with respect to the metric $g_t$. \newline 
Now, comparing Lemma 3.1, Chapter 3, of ~\cite{CK} and (\ref{NRF}), we see that
\[
(\frac{d}{dt}g^{ij})\nabla_i \varphi\nabla_j\varphi = 2\text{Ric} (\nabla\varphi, \nabla\varphi) - r\V\nabla\varphi\V^2.\]
Also note that, \newline$2\int_M \langle \nabla\pa_t\varphi, \nabla \varphi\rangle dV= -2\int_M\pa_t \varphi\Delta\varphi dV = 2\lambda\int_M \varphi\pa_t\varphi dV = \lambda(\int_M \varphi^2 RdV  - r)$.
That gives,
\begin{align*}
\frac{d\lambda}{dt} & = (2\int_M \text{Ric} (\nabla\varphi, \nabla\varphi)dV -r\lambda) + (\lambda\int_M \varphi^2R dV - r\lambda) + \int_M(r - R)|\nabla \varphi|^2dV.
\end{align*}
Using the facts that $\V\nabla \varphi\V^2 = \lambda$, and in dimension $n = 2$, we have $\text{Ric} = \frac{1}{2}Rg$, we have that 
\beq\label{DF}
\frac{d\lambda}{dt} = \lambda\int_M R\varphi^2 dV - r\lambda,
\eeq
which is what we wanted to prove.
\end{proof}
\begin{remark}\label{Sale}
Note that in order to differentiate the eigenfunction $\varphi$ with respect to $t$, as in the proof above, one crucially needs Theorem \ref{Kot}, which ensures that $g_t$ is real analytic with respect to time. If the time variation of the metric is just smooth, the eigenfunctions might not even vary continuously. Heuristically, the problem appears when eigenvalue branches meet (see ~\cite{Ka}).
\end{remark}
Now, consider a surface $M$ of genus $\gamma \geq 2$. By Theorem \ref{Hamil}, we know that the NRF on $M$ flows towards a metric of constant negative curvature. Also, the said constant is actually the average scalar curvature $r$. 
This is part of standard Ricci flow theory. The idea of the proof is first to establish a reaction-diffusion type equation for the evolution of the scalar curvature $R$ under the NRF, and then compare the solution of the said reaction-diffusion equation with the solution of an auxiliary ODE using maximum principles. For more details, see Corollary 5.8 and Lemma 5.9 of Chapter 5 of ~\cite{CK}.
Using such tools, it is established in ~\cite{CK}, Proposition 5.18 that, under the NRF, there is a constant $C > 0$ depending on the initial metric $g_0$ such that 
\beq\label{res}
- Ce^{rt} < R - r < Ce^{rt}.\eeq
But since it seems difficult to obtain an explicit geometric interpretation of this constant $C$, here we want to rewrite (\ref{res}) a bit. Let us have constants $\alpha < \beta < 0$, and consider an NRF family $(M, g_t)$ such that $\alpha < R < \beta$ at time $t = 0$. Then it is clear that $\alpha < r < \beta$. We know that the scalar curvature $R$ on a closed surface under NRF evolves by
\beq\label{evol1}
\pa_t R = \Delta R + R(R - r).\eeq
Let us see what happens to the maximum value of $R$ with respect to time. At time $t$, let $R$ attain a maximum at $p_t$. At $p_t$, by usual multivariable calculus arguments, we have that $\Delta R \leq 0$. That gives us, 
\beq\label{evl}
\frac{dR_{\text{max}}}{dt} \leq R_{\text{max}}(R_{\text{max}} - r).\eeq
Since $R_{\text{max}}$ is negative to start with, and $R_{\text{max}} - r \geq 0$, (\ref{evl}) gives us that $R \leq \beta$ for all space-time.\newline
Pluggin this in (\ref{evol1}), we have that, 
\begin{align*}
\frac{d}{dt}R \leq \beta (R - r) \implies r - R \geq (r - R(0))e^{\beta t}, 
\end{align*}
which implies
\beq\label{FI11}
r - R \geq re^{\beta t}.
\eeq
Also, at the point where $R$ achieves its minimum, we have 
\begin{align*}
\frac{d}{dt}R \geq R(R - r) \geq r(R - r),
\end{align*}
which immediately gives us 
\beq \label{SI}
r - R \leq -\alpha e^{rt}.
\eeq
Using (\ref{DF}), (\ref{FI11}) and (\ref{SI}), we get
\begin{align*}
\alpha \lambda e^{rt} \leq \frac{d\lambda}{dt} \leq -r\lambda e^{\beta t},
\end{align*}
which gives, 
\beq\label{TI}
\alpha e^{rt} \leq \frac{d \text{log}(\lambda)}{dt} \leq -r e^{\beta t}.
\eeq
\subsection{Rearranging the eigenvalues and absolute continuity}\label{dist}
As we move forward in time under the NRF, the eigenvalue branches $\lambda_n (t)$ might intersect. This has the following implication: suppose for each $t$, we arrange the $\lambda_n(t)$ in ascending order, and relabel the spectrum of $g_t$ as $0 \leq \mu_1 (t) \leq \mu_2 (t) \leq ....$ (with multiplicity). Then, these functions $\mu_i(t)$ are of course continuous, but no longer necessarily real analytic, or even differentiable. However, since the functions $\lambda_i (t)$ are real analytic\footnote{We note that the real analyticity of the $\lambda_i$'s is not absolutely essential for the ensuing arguments, but it definitely makes said arguments simpler.}, $\lambda_i (t)$ and $\lambda_j (t)$ can either have discrete intersection when $i \neq j$, or they must be equal for all $t \geq 0$.\newline
So, looking at the $\mu_i (t)$'s, we can see that for a fixed $i$, the function $\text{log}(\mu_i(t))$ is differentiable except at countably many points, with a sharply decreasing derivative, as is given by (\ref{TI}). Observe that the formula (\ref{TI}) works for all eigenvalue branches, so no matter what branch we are on, the absolute value of the derivative of $\text{log}(\mu_i(t))$, outside countably many points, is decreasing exponentially in time. So, each $\text{log}(\mu_i)$ is a Lipschitz function for $t \geq 0$, and hence absolutely continuous, on which we can apply the Fundamental Theorem of Calculus.\newline
Calling $\mu_i (\infty)$ as just $\mu_i$, we can say the following  from (\ref{TI}):
\[
-\frac{\alpha}{r}e^{rt} \leq \text{log}(\frac{\mu_i}{\mu_i(t)}) \leq \frac{r}{\beta}e^{\beta t}, i \geq 1,
\]
which gives us, plugging in $t = 0$, the following:
\beq\label{FI}
e^{\frac{-\alpha}{r}} \leq \frac{\mu_i}{\mu_i(0)} \leq e^{\frac{r}{\beta}}, i \geq 1.
\eeq
Note that in calculating (\ref{FI}), our assumption has been that $M$ has unit area, because this assumption was used to derive the formula in Lemma \ref{2.1}.\newline 
Now, suppose $M$ has area $A$, and at time $t = 0$, its scalar curvature $R$ satisfies $\alpha < R < \beta < 0$. By rescaling $M$ to $\overline{M}$ which has area $1$, we find that the scalar curvature of $\overline{M}$ at time $t = 0$ satisfies $A\alpha < \overline{R} < A\beta < 0$. Also, if the average scalar curvature of $M$ is $r$, then the average scalar curvature of $\overline{M}$ is $\overline{r} = Ar$. Plugging all these in (\ref{FI}) and seeing that the ratio $\frac{\mu_i}{\mu_i(0)}$ remains invariant when the metric on $M$ is scaled, we can now infer that (\ref{FI}) holds regardless of what area $M$ has.
\newline
By Hamilton's result, as mentioned before, a compact surface $M$ of genus $\gamma$ and $r = -2$ will flow towards a hyperbolic surface of the same genus under the NRF. Now, in the Teichmuller space $\Cal{T}_g$, one has a concept of distance (see ~\cite{B}, Definition 6.4.1) as follows:
\begin{definition}\label{dis}
For $S = (S, \varphi)$ and $S' = (S', \varphi') \in \Cal{T}_g$, the distance $\delta$ is defined as 
\[
\delta(S, S') = \inf \text{log } q[\phi],\]
where $\phi$ runs through the quasi isometries $\phi : S \to S'$ in the isotopy class of $\varphi'\circ\varphi^{-1}$ and $q[\phi]$ is the maximal length distortion of $\phi$, that is, $q[\phi]$ is the infimum of all such numbers $q \geq 1$ such that the following holds:
\[
\frac{1}{q}\text{dist }(x, y) \leq \text{dist }(\phi x, \phi y) \leq q\text{ dist }(x, y).\] 
\end{definition}
With that in place, Theorem 14.9.2 of ~\cite{B} states the following:
\begin{theorem}\label{Bus}
Let $S_1, S_2 \in \Cal{T}_g$ be at a distance $\delta$. Then, for any $n \in \NN$, we have
\beq\label{Another}
e^{-4\delta}\lambda_n (S_2) \leq \lambda_n(S_1) \leq e^{4\delta}\lambda_n (S_2).
\eeq
\end{theorem}
Though we are content with merely stating the theorem, the basic idea of the proof is as follows: consider a pair of hyperbolic surfaces $X$ and $Y$ with $\text{dist }(X, Y) = \delta$. To compare eigenvalues, one uses Rayleigh quotients \[ R(f, Y) = \frac{\V \nabla f\V_{L^2(Y)}^2}{\V f\V_{L^2(Y)}^2}, \] 
and uses a given $q$-quasi isometry $\phi : X \to Y$ to prove that  $\V \nabla (f\circ \phi)\V^2 dX \leq q^2 \V \nabla f\V^2 dY$. \newline
Now, let us consider two compact surfaces $M_1$ and $M_2$ of same genus $\gamma \geq 2$, same area and satisfying the same curvature bounds $\alpha < R < \beta < 0$. We can now prove Theorem \ref{MT}, which can be thought of as an extension of Theorem \ref{Bus}.\newline
\begin{proof}
As outlined in the introduction, scale the metric on $M_i$ to produce $N_i$ such that $N_i, i = 1,2$ have average scalar curvature $-2$. Clearly, $\frac{\lambda_n(M_1)}{\lambda_n(M_2)} = \frac{\lambda_n(N_1)}{\lambda_n(N_2)}, n \geq 1$. The claim follows by looking at the ratios $\frac{\lambda_n(N_i)}{\lambda_n(S_i)}, i = 1, 2$ using (\ref{FI}), and then estimating $\frac{\lambda_n(S_1)}{\lambda_n(S_2)}$ using (\ref{Another}).
\end{proof}
\subsection{Acknowledgements} It is a great pleasure to acknowledge the support of MPIM Bonn and the excellent working conditions there. I also thank Werner Ballmann and Henrik Matthiesen for helpful comments.
\bibliographystyle{plain}
\def\noopsort#1{}

\end{document}